\documentclass[12pt]{amsart}

\usepackage{amsfonts,amssymb,enumerate,bm,hhline,makecell,array,mathtools}
\usepackage{mathrsfs}
\usepackage{comment}
\usepackage[normalem]{ulem}
\usepackage[colorlinks=true,citecolor=blue, urlcolor=blue, linkcolor=blue]{hyperref}
\usepackage{tikz-cd}
\usepackage{amscd}
\usepackage[shortlabels]{enumitem}
\setlist[itemize]{noitemsep,nolistsep}
\usepackage{tensor}
\usepackage{tikz}
\usetikzlibrary{matrix,arrows}
\usepackage{extarrows}
\usepackage[toc,page]{appendix}
\usepackage[all,ps,cmtip,rotate]{xy}
\usepackage{color}

\def\cD{\mathscr{D}}
\def\cA{\mathscr{A}}

\def\cF{\mathscr{F}}

\def\cO{\mathscr{O}}

\def\cP{\mathscr{P}}
\def\cH{\mathscr{H}}
\def\cE{\mathscr{E}}

\def\cS{\mathscr{S}}

\def\cN{\mathscr{N}}
\def\cK{\mathscr{K}}
\def\cT{\mathscr{T}}

\def\bv{\ensuremath{\mathbf v}}

\def\Z{{\bf Z}}

\def\C{{\bf C}}

\def\R{{\bf R}}
\def\Q{{\bf Q}}
\def\P{{\bf P}}

\def\phi{\varphi}

\DeclareMathOperator{\alb}{alb}

\DeclareMathOperator{\Br}{Br}
\DeclareMathOperator{\ch}{ch}

\DeclareMathOperator{\coh}{coh}

\DeclareMathOperator{\Db}{D\textsuperscript{\rm b}}

\DeclareMathOperator{\Hilb}{Hilb}

\def\Im{\mathop{\rm Im}\nolimits}

\DeclareMathOperator{\Kum}{Kum}

\DeclareMathOperator{\NS}{NS}

\DeclareMathOperator{\Pic}{Pic}

\DeclareMathOperator{\Stab}{Stab}

\def\eps{\varepsilon}

\def\isom{\simeq}
\def\cong{\isom}

\def\lra{\longrightarrow}
\def\llra{\hbox to 10mm{\rightarrowfill}}
\def\lllra{\hbox to 15mm{\rightarrowfill}}

\def\llla{\hbox to 10mm{\leftarrowfill}}
\def\lllla{\hbox to 15mm{\leftarrowfill}}

\DeclareMathOperator{\isomto}{\stackrel{{}_{\scriptstyle\sim}}{\to}}

\DeclareMathOperator{\isomlra}{\stackrel{{}_{\scriptstyle\sim}}{\lra}}

\newtheorem{lemm}{Lemma}[section]
\newtheorem{theo}[lemm]{Theorem}
\newtheorem*{MainThm*}{Main Theorem}
\newtheorem{coro}[lemm]{Corollary}
\newtheorem{prop}[lemm]{Proposition}

\theoremstyle{remark}
\newtheorem{defi}[lemm]{Definition}
\newtheorem{rema}[lemm]{Remark}

\newtheorem{exam}[lemm]{Example}

\makeatletter
\@namedef{subjclassname@2020}{%
  \textup{2020} Mathematics Subject Classification}
\makeatother

\makeatletter
\def\@tocline#1#2#3#4#5#6#7{\relax
  \ifnum #1>\c@tocdepth 
  \else
    \par \addpenalty\@secpenalty\addvspace{#2}%
    \begingroup \hyphenpenalty\@M
    \@ifempty{#4}{%
      \@tempdima\csname r@tocindent\number#1\endcsname\relax
    }{%
      \@tempdima#4\relax
    }%
    \parindent\z@ \leftskip#3\relax \advance\leftskip\@tempdima\relax
    \rightskip\@pnumwidth plus4em \parfillskip-\@pnumwidth
    #5\leavevmode\hskip-\@tempdima
      \ifcase #1
       \or\or \hskip 1em \or \hskip 2em \else \hskip 3em \fi%
      #6\nobreak\relax
    \dotfill\hbox to\@pnumwidth{\@tocpagenum{#7}}\par
    \nobreak
    \endgroup
  \fi}
\makeatother

\def\id{\mathsf{id}}

\title[Complete curves of polarized K3 surfaces and hyper-K\"ahler manifolds]{Complete curves in the moduli space of polarized K3 surfaces and hyper-K\"ahler manifolds}

\begin{document}

\author{Olivier Debarre}
\address{\parbox{0.9\textwidth}{Universit\'e de Paris and Sorbonne Universit\'e\\[1pt]
CNRS, IMJ-PRG\\[1pt]
F-75013 Paris, France\vspace{1mm}}}
\email{{olivier.debarre@imj-prg.fr}}

\author{Emanuele Macr\`i}
\address{\parbox{0.9\textwidth}{Universit\'e Paris-Saclay\\[1pt]
CNRS, Laboratoire de Math\'ematiques d'Orsay\\[1pt]
Rue Michel Magat, B\^at. 307, 91405 Orsay, France
\vspace{1mm}}}
\email{{emanuele.macri@universite-paris-saclay.fr}}

\thanks{This work was partially supported by the ERC Synergy Grant ERC-2020-SyG-854361-HyperK}

\begin{abstract}
Building on an idea of Borcherds,  Katzarkov,  Pantev, and Shepherd-Barron (who treated the case $e=14$), we prove that the moduli space of polarized K3 surfaces of degree $2e$ contains complete curves for all $e\geq 62$ and for some sporadic lower values of $e$ (starting at~$14$).
We also construct complete curves in the moduli spaces of polarized hyper-K\"ahler manifolds of $\mathrm{K3}^{[n]}$-type or $\mathrm{Kum}_n$-type for all $n\ge 1$ and polarizations of various degrees and divisibilities.
\end{abstract}

\maketitle

\tableofcontents
\setcounter{tocdepth}{1}

\begin{flushright}
{\em To the memory of Alberto Collino}
\end{flushright}


\section{Introduction}\label{sec:intro}

Let $\cF^0_{2e}$ be the ($19$-dimensional irreducible quasi-projective) moduli space of {\em polarized} K3 surfaces of degree $2e$.\
In~\cite[Theorem~1.3]{bkps}, the authors prove that $\cF^0_{2}$ is affine, hence contains no complete curves.\
In~\cite[Section~3]{bkps}, they construct a complete curve in $\cF^0_{28}$.\
Their idea is to start from a nonisotrivial family of polarized abelian surfaces, which exists because the moduli space of polarized abelian surfaces has a small boundary in its Satake compactification, and construct a suitable polarization on the associated family of Kummer surfaces (which needs to be ample on {\em all} Kummer surfaces in the family).\
Using the same construction, we prove the following extension of their results, which partially answers a question asked in \cite[\S~6.3.4]{ben}.

\begin{theo}\label{main}
For each integer $e\ge 62$ or in the set
\[
\{14,18, 26, 28, 29, 32, 34, 36, 38, 40, 42, 44, 45, 46, 47,   49, 50, 53, 54, 56, 57, 59, 60\},
\]
there exists a complete curve in $\cF^0_{2e}$ or, equivalently, a nonisotrivial smooth family of  K3 surfaces with a relative polarization of degree $2e$.
\end{theo}

As explained in Remark~\ref{rem}, one can make the construction quite explicit and obtain for example, for infinitely many values of $e$ (including $e=14$), complete rational curves in $\cF^0_{2e}$  defined over $\Q$.\ The Picard numbers of  the corresponding  K3 surfaces are  $19$ or $20$.

It is plausible that $\cF^0_{2e}$ contain complete curves for all $e\ge 2$.\ The situation is very different for the moduli space 
$\cF_{2e}$ of {\em quasi-polarized} K3 surfaces of degree $2e$:\footnote{The moduli space  $\cF^0_{2e}$ is an open subset of $\cF_{2e}$ which is the complement in the period space of a Heegner divisor; this Heegner divisor has one or two irreducible components depending on whether $e\not\equiv 1\pmod4$ or $e\equiv 1\pmod4$.} because of the existence of its Baily--Borel projective compactification with one-dimensional boundary,~$\cF_{2e}$ contains complete subvarieties of dimension $17$.\ This is known to be the maximal possible dimension (\cite[Corollary~4.3]{gk}).

In the last part of the paper, we construct complete curves in the moduli spaces of polarized hyper-K\"ahler manifolds of $\mathrm{K3}^{[n]}$-type and $\Kum_n$-type by using moduli space techniques as in~\cite{muk1,muk2}.\
Given a K3 or  abelian surface (or more generally a smooth and proper CY2 category) with a Bridgeland stability condition and a Mukai vector, there is a natural polarization on the corresponding moduli space~\cite{bama1} of stable objects which behaves well in families.\
Hence, the question of finding a suitable polarization on the moduli space translates into the question of understanding stable objects, which can be studied effectively by using techniques from~\cite{bama2,myy,yy,yos3}.

The final result is less complete than what can be obtained with the approach described above in the surface case, but it covers more cases.\
We prove the existence of a complete curve in $\cF^0_{2e}$ for $e\in\{18,32,36,50,54\}$ (see Example~\ref{ex:PolarizedK3}; we could not obtain these values cannot be obtained directly with the previous technique), while in higher dimensions, our results imply the following two theorems  for Hilbert schemes of points on K3 surfaces and generalized Kummer varieties (for more general statements, see Propositions~\ref{prop:AmpleHilbert} and~\ref{prop:AmpleKummer}). 

In what follows, we denote by $T$ a quasi-projective scheme and we let $n\geq2$.

\begin{theo}\label{thm:Hilb}
Let $\cF \to T$ be a smooth family of K3 surfaces with a relatively ample divisor~$H_{\cF}$.\ Let $H_{\cF,n}$ be the big and nef divisor on $\Hilb^n(\cF/T)$ induced by $H_{\cF}$ via the symmetric product.\ Assume that the class of the divisor on $\Hilb^n(\cF/T)$ parameterizing nonreduced subschemes is divisible by 2 and let  $\delta_{\cF}$ be a half.\
For  all positive integers $a,b$ such that $a > b \sqrt{(n-1)^2+4(n-1)}$, the divisor
\[
a  H_{\cF,n} - b  \delta_{\cF}
\]
is relatively ample on $\Hilb^n(\cF/T) \to T$.
\end{theo}

In particular, given a complete curve in $\cF^0_{2e}$ and relatively prime positive integers $a,b$ such that    $a > b \sqrt{(n-1)^2+4(n-1)}$, after a finite base change, we obtain a complete curve in the moduli space of polarized hyper-K\"ahler manifolds of $\mathrm{K3}^{[n]}$-type of degree $2ea^2-2b^2(n-1)$ and divisibility $\gcd(a,2(n-1))$.

\begin{theo}\label{thm:Kum}
Let $\cA \to T$ be a smooth family of abelian surfaces with a relatively ample divisor~$H_{\cF}$.\ Assume that the class of the divisor on $\Kum_n(\cA/T)$ parameterizing nonreduced subschemes is divisible by 2 and let  $\delta_{\cA}$ be a half.\
For all positive integers $a,b$ such that $a > b(n+1)$, the divisor
\[
a  H_{\cA,n} - b  \delta_{\cA}
\]
is relatively ample on $\Kum_n(\cA/T) \to T$.
\end{theo}

The notation $H_{\cA,n}$  means, as before, the restriction  to $\Kum_n(\cA/T)$ of the corresponding divisor  on the Hilbert scheme.

In particular, given a complete curve of primitively polarized abelian surfaces of degree~$2d$ and positive integers $a,b$  such that $\gcd(a,b)=1$  and $a > b(n+1)$,  we obtain, after a finite base change, a complete curve in the moduli space of polarized hyper-K\"ahler manifolds of $\Kum_{[n]}$-type of degree $2da^2-2b^2(n+1)$ and divisibility $\gcd(a,2(n+1))$.

\subsection*{Acknowledgements}
We would like to thank Ignacio Barros, Daniel Huybrechts, Kieran O'Grady, and Claire Voisin for useful discussions and suggestions.\ Many thanks also to the referee who helped us improve the exposition.


\section{Ample classes on Kummer surfaces}\label{sec:AmpleKummer}

Let $A$ be an abelian surface and let $\eps\colon  \widehat A\to A$ be the blow up of the sixteen 2-torsion points of~$A$.\ Let $\Kum(A)=\widehat A/\pm1$ be the {\em Kummer surface} of $A$, with quotient map $\pi\colon \widehat A\to  \Kum(A)$, and let $E_1,\dots,E_{16}$ be the images   by $\pi$ of the exceptional curves of $\eps$; each~$E_i$ is a rational curve with self-intersection~$-2$.\ By~\cite[Propositions~VIII.(5.1) and~VIII.(5.2)]{bpv}, there is an injective morphism of Hodge structures
$$\alpha\coloneqq \pi_*\eps^* \colon H^2(A,\Z)\lra H^2(\Kum(A),\Z)$$
and it satisfies, for all $x,y\in H^2(A,\Z)$,
$$\alpha(x)\cdot \alpha (y)=2 x\cdot y.$$   
In particular, $\alpha (x)^2\in 4\Z$.\ Moreover, $\Im(\alpha)$ is the orthogonal complement in $H^2(\Kum(A),\Z)$ of the sublattice generated by (the classes of) $E_1,\dots,E_{16}$ (\cite[Corollary~VIII.(5.6)]{bpv}).

Any integral divisor class on $\Kum(A)$ can therefore be written as
\begin{equation}\label{dec}
L=a\alpha(N)-\sum_{i=1}^{16}a_iE_i,
\end{equation}
where $  N $ is primitive in $\NS(A)$ and $a,a_1,\dots,a_{16}\in \Q$.\ By taking the product with $E_i$, we get $a_i\in\frac12\Z$.\ By taking the product with $\alpha(x)$, where $x\in H^2(A,\Z)$ is such that $N\cdot x=1$, we get $a\in\frac12\Z$.\  
The $a_i$  are not necessarily integers: for example, the class $\frac12\sum_{i=1}^{16}E_i$ is integral.

Let $H_A$ be an ample divisor class on $A$ of square $2d$, for some $d\in\Z_{>0}$.\ The divisor class~$H\coloneqq \alpha(H_A)$ on  $\Kum(A)$ is  nef and big, but not ample; one has $H^2=4d$.

We will from now on drop $\alpha$ in~\eqref{dec}.\ 
We look for conditions for a class~$L$ as in~\eqref{dec} to be  ample.\
Assume $L^2>0$ and $0<H\cdot L=aH\cdot N$ (so for example,~$a>0$ and~$H\cdot N>0$), so that~$L$ is in the positive cone; we have then (\cite[Proposition~2.1.4]{huy}) an equivalence
\[
L \textnormal{ ample } \Longleftrightarrow\ L\cdot C>0 \textnormal{ for all irreducible curves $C\subset \Kum(A)$ with } C^2=-2. 
\]
The next proposition gives sufficient conditions for a class $L$ as in~\eqref{dec} to be ample.\
Its proof follows that of \cite[Proposition~4.4]{gs2} (itself identical to that of \cite[Proposition~3.4]{gs}), where similar results are proven under the additional assumption $\Pic(A)=\Z H_A$.

\begin{prop}\label{prop1}
On any Kummer surface, any rational  class
\[
L=aH- \sum_{i=1}^{16}a_iE_i,
\]
where  $a_1\ge \dots\ge a_{16}>0$ and such that $a>  a_{1}+a_{2}+a_{3}+a_{4}$, is ample.
\end{prop}

\begin{proof}
We have
\begin{equation}\label{ha}
a^2>(a_1+a_2+a_3+a_4)^2=\sum_{1\le i,j\le 4}a_ia_j\ge  \sum_{i=1}^{16}a_i^2,
\end{equation}
hence
$$ L^2= 4 a^2d-2\sum_{i=1}^{16}a_i^2>(4d-2)a^2>0 $$
and $L\cdot H=4ad>0$, so that $L$ is in the positive cone.\ Consider now an irreducible curve $C$ with self-intersection $-2$; we must prove  $L\cdot C>0$.\ One can  write  
\begin{equation}\label{C}
    C=bM-\sum_{i=1}^{16}b_iE_i,
\end{equation}
with $b,b_1,\dots,b_{16}\in\frac12\Z$.\ Since $L\cdot E_i=2a_i>0$, we  may assume that $C$ is not one of the    $E_i$.\ We then have $b_i=\frac12 C\cdot E_i\ge 0$ and $0< H\cdot C=b H\cdot M$, so we may assume $b>0$ and  $H\cdot M >0$.\ 
Finally, the  integer $M^2$ is divisible by $4$ and nonnegative, because  $2bM$ is the class of the effective divisor  $\eps_*\pi^*C$  on $A$.\footnote{This is an essential point, and the only one where we use the geometry of the situation.\ The rest is just formal lattice-theoretic manipulations.}
  
Assume by way of contradiction $L\cdot C\le0$, that is,
\begin{equation}\label{hm}
abH\cdot  M -2\sum_{i=1}^{16}a_ib_i\le 0.
\end{equation}
Following \cite{gs}, we use the   inequalities
$$\Bigl( \sum_{i=1}^{16}b_i\Bigr)^2\le \Bigl(\sum_{i=1}^{16}a_i^2\Bigr)  \Bigl( \sum_{i=1}^{16}b_i^2\Bigr)\ \ \textnormal{ (Cauchy--Schwarz)} \quad,\quad (H\cdot  M)^2\ge H^2M^2\ \ \textnormal{ (Hodge)}.
$$
They imply
$$
\begin{aligned}
2= -C^2& =-b^2M^2+2\sum_{i=1}^{16}b_i^2\\
&\ge - b^2  M^2+\frac{2 \Bigl( \sum_{i=1}^{16}a_ib_i\Bigr)^2}{ \sum_{i=1}^{16}a_i^2} \qquad&\textnormal{(use Cauchy--Schwarz)}\\
&\ge  -b^2M^2+ \frac{a^2b^2 (H\cdot  M)^2}{2 \sum_{i=1}^{16}a_i^2}\qquad&\textnormal{(use~\eqref{hm})}\\
&\ge  -b^2 M^2 +\frac{a^2b^2 H^2M^2}{2 \sum_{i=1}^{16}a_i^2} \qquad&\textnormal{(use Hodge)}\\
&= b^2 M^2\biggl(\frac{2  a^2d}{ \sum_{i=1}^{16}a_i^2}-1\biggr).
\end{aligned}
$$
If $b^2M^2\ge 2$, we obtain $ a^2d\le\sum_{i=1}^{16}a_i^2 $, which contradicts~\eqref{ha}.

Assume now $b^2M^2< 2$.\ Since $M^2$ is divisible by $4$ and nonnegative and $b\in\frac12\Z_{>0}$, this leaves only two cases to be considered:
\begin{enumerate}[(A)]
\item   $b=\frac12$ and $M^2=4$.\ Using~\eqref{C} and $C^2=-2$, we obtain $\sum_{i=1}^{16}b_i^2=\frac34$.\ 
Since $b_i\in\frac12\Z_{\ge0}$, we  get  either $b_{i_1}=1$, $b_{i_2}=b_{i_3}=\frac12$, or $b_{i_1}=\dots=b_{i_6}=\frac12$, and the others vanish.\  Plugging that back into~\eqref{hm}, we get that $abH\cdot  M $ is at most $2a_{i_1}+a_{i_2}+a_{i_3}$ in the first case, and at most $a_{i_1}+\dots+a_{i_6}$ in the second case.\  The Hodge inequality also gives $bH\cdot M\ge \sqrt{dM^2}\ge 2$.\ Therefore, we get $a\le \frac{1}{2}(2a_{i_1}+a_{i_2}+a_{i_3})$ in the first case and  $a\le\frac{1}{2}(a_{i_1}+\dots+a_{i_6})$ in the second case, which contradicts  our  hypothesis on $a$.
\item $M^2=0$ (but $M\ne 0$).\ Using~\eqref{C} and $C^2=-2$, we get
$\sum_{i=1}^{16}b_i^2=1$, from which we deduce that either $b_{i_1}=1$, or $b_{i_1}=\dots=b_{i_4}=\frac12$, and the others vanish.\ 
Using~\eqref{hm}, we get that $abH\cdot  M $ is at most $2a_{i_1}$ in the first case, and at most $a_{i_1}+\dots+a_{i_4}$ in the second case.\ Since  the linear system $|H|$ is base-point-free and only contracts the curves  $E_1,\dots,E_{16}$, the linear system $|H-E_{i_1}|$ has no fixed divisor, hence is nef; in particular, $0\le (H-E_{i_1})\cdot C = bH\cdot M-2b_{i_1} $.\ Therefore, we get $a\le a_{i_1} $ in the first case and  $a\le a_{i_1}+\dots+a_{i_4} $ in the second case, which contradicts  our  hypothesis on $a$.
\end{enumerate}
We have therefore obtained  contradictions in all cases.\ This proves  $L\cdot C>0$ for all $(-2)$-curves~$L$, hence   $L$ is ample.
\end{proof}


\begin{coro}\label{prop2}
On any Kummer surface, the  class $aH- \sum_{i=1}^{16}E_i$
is ample for all~$a>4$, and nef for $a=4$.
\end{coro}

This seems to have been known, at least when $a$ is an integer, to the authors of \cite{bkps} (see their Section~3) but we could not find a published proof.\

\begin{rema}\label{p0}
Let $A_1$ and $A_2$ be elliptic curves.\ The ample class $H_A\coloneqq (A_1\times \{0\})+d(\{0\}\times A_2)$ on the abelian  surface $A\coloneqq A_1\times A_2$ induces a  polarization $H$ of degree $4d$ on the Kummer surface $\Kum(A)$.\ 
The image in $\Kum(A)$ of the proper transform $\eps^*(\{0\}\times A_2)- \sum_{i=1}^{4}E_i$ in~$\widehat A$ of the elliptic curve $\{0\}\times A_2$ is a  smooth rational curve with self-intersection $-2$ whose intersection number with $4H- \sum_{i=1}^{16}E_i$ is $0$.\
The nef class $4H- \sum_{i=1}^{16}E_i$   is therefore not ample on $\Kum(A)$.
\end{rema}

\begin{rema}\label{p2}
If we make the additional generality assumption $\Pic(A)=\Z H_A$, we can take \mbox{$M= H$} in the proof of  Proposition~\ref{prop1}.\ In particular, the case  $M^2=0$ needs not be considered and the hypotheses can be relaxed.\ One checks for example that when $d=1$, the class $aH-\sum_{i=1}^{16}E_i$ is ample for all  $a>3$.\footnote{Assume $a>3$.\ 
If $b\ge 2$, we  get the contradiction $2\ge b^2 M^2\bigl(\frac{2  a^2d}{ \sum_{i=1}^{16}a_i^2}-1\bigr)\ge 2a^2-16>2$.\ 
If  $b=\frac12$, we get $\sum_{i=1}^{16}b_i^2=\frac32$, hence $\sum_{i=1}^{16}b_i\le  3$, hence $a\le \frac{1}{2b}\sum_{i=1}^{16}b_i\le  3$.\ 
If  $b=1$, we get $\sum_{i=1}^{16}b_i^2=3$, hence $\sum_{i=1}^{16}b_i\le  6$, hence $a\le \frac{1}{2b}\sum_{i=1}^{16}b_i\le  3$.\ If  $b=\frac32$, we get $\sum_{i=1}^{16}b_i^2=\frac{11}2$, hence $\sum_{i=1}^{16}b_i\le  9$, hence $a\le \frac{1}{2b}\sum_{i=1}^{16}b_i\le  3$.}\ It is   classical that when $(A,H_A)$ is an indecomposable principally polarized abelian surface (that is, the Jacobian of a smooth curve), the integral degree-$8$ class $2H-\frac12\sum_{i=1}^{16}E_i$ is   very ample (\cite[p.~773--787]{gh}).\ See also \cite[Section~4]{gs2} for other results under the assumption $\Pic(A)=\Z H_A$.
\end{rema}


\section{Degrees of ample classes on Kummer surfaces}\label{sec:DegreeAmpleKummer}

Our aim in this section is to prove that all large enough even numbers can be realized as degrees of primitive ample   classes on Kummer surfaces.\ More precisely, we prove in Section~\ref{subsec:PrimitiveAmpleNot8} and Section~\ref{subsec:PrimitiveAmple4} the following result.

\begin{prop}\label{prop20}
Let $A$ be a principally polarized surface.\ For every   integer $e\ge 160$, or even $\ge 60$, or in the set $\{
14$, $26$, $28$, $34$, $40$, $42$, $44$, $46$, $53$, $56$, $79$, $97$, $101$, $103$, $107$, $109$, $113,119$, $125$, $131$, $135$, $137$, $139$, $143$, $145$, $149$, $151$, $155$, $157\}$, there exists a primitive ample class of degree $2e$ on $\Kum(A)$.
\end{prop}

In the remainder of this section, $(A,H_A)$ will be a principally polarized surface.\ The ample class~$H_A$ induces on the surface $\Kum(A)$ a nef class $H$  with self-intersection 4.
 
\subsection{Primitive ample classes of large degrees not divisible by 8}\label{subsec:PrimitiveAmpleNot8}

Keeping the   notation of Section~\ref{sec:AmpleKummer}, we first consider integral classes on $\Kum(A)$ of the form
$$L=aH- \sum_{i=1}^{16}a_iE_i ,$$
where $a,a_1,\dots,a_{16}$ are  integers, $a_1\ge \dots\ge a_{16}>0$, and $a>a_{1}+a_{2}+a_{3}+a_{4}$.\ By Proposition~\ref{prop1}, these classes are ample.\ 
We have
$$L^2=4a^2- 2\sum_{i=1}^{16}a_i^2  $$
and we need to examine which positive even integers $2e$ can be written in this way.\ Since we are looking for primitive classes, it will help to take $a_{16}=1$: we then have $L\cdot E_{16}=2$, hence either~$L$ is primitive, or it is divisible by $2$, in which case $L^2=2e$ must be divisible by $8$ (because the intersection pairing is even).\ It is easy to see that restricting to the case  $a_{16}=1$ does not narrow the search.
 
\begin{prop}\label{prop21}
Any integer $e\ge 163$ or in the set $\{34$,  $53$, $79$, $97$, $101$, $103$, $107$, $109$, $113,119$, $125$, $131$, $135$, $137$, $139$, $143$, $145$, $149$, $151$, $155$, $157$, $161\}$
can be written as
\begin{equation}\label{eq1}
e=2 a^2 -\sum_{i=1}^{15}a_i^2-1, 
\end{equation}
where $a,a_1,\dots,a_{15}$ are  integers, $a_1\ge \dots\ge a_{15}\ge1$, and $a> a_{1}+a_{2}+a_{3}+a_{4}$.
\end{prop}

\begin{proof}
We will first prove the result for all $ e\ge 1216$.\ We begin with a classical result (\cite[Theorem~5.6]{nz}; see also \cite[Sequence~A047700]{oeis}).
 
\begin{lemm}\label{l1}
Every integer $m\ge 34$ is the sum of five positive perfect squares.
\end{lemm}

\begin{lemm}\label{l2}
Every integer $m\ge 36$ is the sum of the squares of fifteen positive integers that are all $\le \sqrt{\frac{m}3-3}$.
\end{lemm}
 
\begin{proof}
Write $m=3n+r$, with $r\in\{0,1,2\}$.\ If $m\ge 102$, we have $n\ge 34$.\ By Lemma~\ref{l1}, we can write both $n$ and $n+1$ as the sum  of the squares of five positive  integers that are all $\le \sqrt{n-3}\le \sqrt{\frac{m}3-3}$.\ Adding up these decompositions, we obtain the required decomposition of $m$.

The remaining cases $36\le m\le 101$ can be checked  by computer (see~\ref{st1}).
\end{proof}

A consequence  of Lemma~\ref{l2} is that all integers $e+1$ between $2a^2-36$ and $2a^2-\lfloor\frac{3a^2+143}{16} \rfloor$ can be written as $e+1=2 a^2 -\sum_{i=1}^{15}a_i^2$, where the $a_i$ are integers such that 
$$
1\le a_i\le \sqrt{\frac{1}3\Bigl( \frac{3a^2+143}{16}\Bigr)-3}<\frac{a}{4}.
$$
In order to have no  gap when we go from $a$ to $a+1$, we need
$$
2(a+1)^2-\Big\lfloor\frac{3(a+1)^2+143}{16} \Big\rfloor\le 2a^2-35
$$
or, equivalently,
$$
\frac{3(a+1)^2+143}{16} \ge 4a+37,
$$
that is, $a\ge 26$.\
This means that all integers $e $ such that
$$
e+1\ge 2\cdot 26^2-\Big\lfloor\frac{3\cdot 26^2+143}{16} \Big\rfloor=1217
$$
can be written as in~\eqref{eq1} with $a>4a_1$, which is more than we need for Proposition~\ref{prop21}.\ 

The remaining cases can be checked by computer (see~\ref{st2}).
\end{proof}

\begin{coro}\label{cor25}
Let $A$ be a principally polarized surface.\
For all integers $e$ as in Proposition~\ref{prop21}, there exists a primitive ample class of degree $2e $ on $\Kum(A)$.
\end{coro}

\begin{proof}
Write the integer $e$   as in~\eqref{eq1}.\ 
By Proposition~\ref{prop1}, the class $aH- \sum_{i=1}^{15}a_iE_i-E_{16}$ is then ample on $\Kum(A)$, of degree $2e$.\ As explained earlier, the condition $4\nmid e$ ensures that it is primitive.
\end{proof}
 
We left out some even values of $e$ that can also be written as in~\eqref{eq1} because Corollary~\ref{cor29} will give another way to reach them.

\subsection{Primitive ample classes of large degrees  divisible by 4}\label{subsec:PrimitiveAmple4}

Keeping our  principally polarized surface $(A,H_A)$ and the same notation, recall that the class $\sum_{i=1}^{16}E_i$ is (uniquely) divisible by $2$ in $\NS(\Kum(A))$ (because it is the class of the branch locus of the double cover $\pi\colon \widehat A\to  \Kum(A)$).\  
In this section, we consider integral  classes of the form
$$L=aH- \sum_{i=1}^{16}\Bigl(a_i-\frac12\Bigr)E_i $$
where $a,a_1,\dots,a_{16}$ are  integers, $a_1\ge \dots\ge a_{16}\ge1$, and $a>a_{1}+a_{2}+a_{3}+a_{4}-2$.\ 
By Proposition~\ref{prop1}, these classes are ample.\ As before, we will take $a_{16}=1$, which will ensure that the class $L$ is primitive, because $L\cdot E_{16}$ is then equal to $1$.\ 
We have
\begin{align*}
L^2&=a^2H^2+\sum_{i=1}^{15}\Bigl(a_i-\frac12\Bigr)^2(-2)+\frac14(-2)\\
&=4a^2-2\sum_{i=1}^{15}\Bigl(a_i^2-a_i+\frac14\Bigr)^2-\frac12\\
&=4a^2-8-4\sum_{i=1}^{15}\binom{a_i}{2}
\end{align*}
and we need to examine which positive even integers can be written in this way.

\begin{prop}\label{prop26}
Every integer $n\ge 32$ or in the set $\{
9$, $15$, $16$, $22$, $23$, $24$, $25$, $30\}$ can be written as
\begin{equation}\label{eq7}
n=a^2 -\sum_{i=1}^{15}\binom{a_i}{2}, 
\end{equation}
where $a,a_1,\dots,a_{15}$ are  integers, $a_1\ge  \dots\ge a_{15}\ge1$, and $a> a_{1}+a_{2}+a_{3}+a_{4}-2$.
\end{prop}

\begin{proof}
We will first prove the result for all $n\ge 187$.\ Recall that a triangular number is an integer of the form $\binom{r}{2}$, with $r\in\Z$.
 
\begin{lemm}[Gauss' Eureka Theorem]\label{lem25}
Every nonnegative integer is the sum of three nonnegative triangular numbers.
\end{lemm}

\begin{lemm}\label{lem27}
Every integer $m\ge 24$ is the sum of fifteen nonnegative triangular numbers of the form $\binom{a}{2}$ with $1\le a\le \frac12+\sqrt{\frac{2m}5+\frac{33}4}$.
\end{lemm}
 
\begin{proof}
Write $m=5n+r$, with $r\in\{0,\dots,4\}$ and $n\ge0$.\
By Lemma~\ref{lem25}, we can write $n,\dots,n+4$ each as the sum  of three nonnegative triangular numbers~$\binom{a}{2}$.\ 
We have $\binom{a}{2}\le n+4\le\frac{m}{5}+4$, hence $a\le \frac12+\sqrt{\frac{2m}5+\frac{33}4}$.\ 
Adding up these decompositions, we obtain the required decomposition of $m$.
\end{proof}

A consequence of Lemma~\ref{lem27} is that all integers between $a^2$ and $a^2-\lfloor \frac{5a^2-661}{32} \rfloor$ can be written as in~\eqref{eq7}, where the $a_i$ are integers such that 
$$
1\le  a_i\le   \frac12+\sqrt{\frac{2}5\Bigl( \frac{5a^2-661}{32} \Bigr) +\frac{33}4}<\frac{a+2}{4}.
$$
In order to have no  gap when we go from $a$ to $a+1$, we need
$$
(a+1)^2-\Bigl\lfloor\frac{5(a+1)^2-661}{32} \Big\rfloor\le a^2+1
$$
or, equivalently,
$$
\frac{5(a+1)^2-661}{16}\ge 2a,
$$
that is, $a\ge 14$.\
This means that every integer
$$
n\ge  14^2-\Big\lfloor  \frac{5\cdot14^2-661}{32} \Big\rfloor=187 
$$
can be written as in~\eqref{eq7} with $a>4a_1-2$, which is more than we need.
 
The remaining cases where $n\le 186$  can be checked by computer (see~\ref{st3}).\ 
This proves Proposition~\ref{prop26}.
\end{proof}
 
\begin{coro}\label{cor29}
Let $A$ be a principally polarized surface.\ 
For every even  integer $e\ge 60$ or in the set $\{
14$, $26$, $28$, $40$, $42$, $44$, $46$, $56\}$, there exists a primitive ample class of degree $2e$ on~$\Kum(A)$.
\end{coro}

\begin{proof}
Write $n\coloneqq \frac12 e +2$  as  in Proposition~\ref{prop26}.\ 
By Proposition~\ref{prop1}, the primitive integral class 
$ aH- \sum_{i=1}^{15}\bigl(a_i-\frac12\bigr)E_i-\frac12 E_{16}$ 
is ample on $\Kum(A)$, of degree $2e$.\end{proof}

Proposition~\ref{prop20} is then just   Corollary~\ref{cor25} and  Corollary~\ref{cor29} put together.

\subsection{Ample classes of intermediate degrees}\label{sect23}

We keep  our  principally polarized surface $(A,H_A)$ and the same notation.\ 
Not only is the class $\sum_{i=1}^{16}E_i$ divisible by $2$ in $\NS(\Kum(A))$, but this is also true for  some sums of eight classes among the $E_i$.\footnote{\label{f2}This fact  can be explained geometrically as follows.\ Write $A=A'/\{0,\alpha\}$, where $A'$ is an abelian surface and $\alpha\in A'$ has order 2.\ The translation by~$\alpha$ on~$A'$ induces  an involution on $\Kum( A')$ whose fixed points are  the eight images in $\Kum( A')$ of the sixteen points $x\in A$ such that $2x=\alpha$.\ Blowing up these points, we get a double cover of $\Kum(A)$ branched along the   union of the eight $(-2)$-curves corresponding to the images in $A$ of the points $x$.\ The sum of these eight $(-2)$-curves is therefore divisible by $2$.} 

\noindent (M1) We may consider primitive integral  classes of the form
$$aH- \sum_{i=1}^{15}a_iE_i-\frac12E_{16}$$
where $a$ is an integer, $a_1,\dots,a_{15}\in \frac12\Z_{>0}$  and exactly seven of them are not  integers, $a_1\ge \dots\ge a_{15}>0$, and $a>a_{1}+a_{2}+a_{3}+a_{4}$ (\cite[Remark~2.3]{gs2}).\ 
By Proposition~\ref{prop1}, these classes are ample.\  
For new possible degrees $2e$, we get all values $95\le e\le 159$ and $e\in\{38$, $57$, $59$, $71$, $73$, $75$, $77$, $79$, $81$, $83$, $85 \}$ (see~\ref{st4}).
 
\noindent (M2) There are integral classes of the type $\frac12(H+  E_{i_1}+\dots+ E_{i_6})$ (see \cite[Theorem~2.7]{gs2}\footnote{When   the principally polarized surface $(A,H_A)$ is indecomposable, the image of $\Kum(A)$ by the morphism associated with the linear system $|H|$ is a quartic surface in $\P^3$ with sixteen nodes and sixteen everywhere tangent planes (classically called ``tropes''), each containing six nodes.\ Each trope gives rise to a conic with class of the type $\frac12( H-E_{i_1}-\dots- E_{i_6})$.\ When the surface $(A,H_A)$ is the product of  principally polarized elliptic curves  $A_1$ and $ A_2$, these sixteen conics become unions of copies  of $A_1$ and $A_2$ meeting transversely at one point (\cite[Section~3.1]{dole}). 
})
so we may also consider primitive integral  classes of the form
$$\Bigl(a+\frac12\Bigr)H- \sum_{i=1}^{15}a_iE_i-\frac12E_{16}$$
where $a$ is an integer, $a_1,\dots,a_{15}\in \frac12\Z_{>0}$  and exactly five of them are not  integers, $a_1\ge \dots\ge a_{15}>0$, and $a\ge a_{1}+a_{2}+a_{3}+a_{4}$.\ 
By Proposition~\ref{prop1}, these classes are ample.\ 
We get possible degrees $2e$ for  the new values
$
e\in\{29$, $ 45$, $  47$, $ 49$, $ 63$, $ 65$, $ 67$, $ 69$, $87$, $89$, $91$, $93 \}
$ (see~\ref{st5}).
 
All in all, in our list of possible degrees $2e$, we get for $e$ the values 
$$
14,26,28,29, 34, 38, 40, 42,44,45, 46, 47,   49, 53, 56, 57,  59, 60
$$
and all integers $e\ge 62$.\
The remaining values $e\in\{18,32,36,50,54\}$, which are needed for the proof below, will be covered later in Example~\ref{ex:PolarizedK3}.

\begin{proof}[Proof of Theorem~\ref{main}]
Let $T$ be an irreducible complete curve contained in the ($3$-dimensional) coarse moduli space $\cA_2$ of principally polarized abelian surfaces; such curves exist because the boundary of the ($3$-dimensional) projective Satake compactifications  has dimension $1$.
 
We fix $m$ even, $m\ge 4$, and  we consider the pullback $T'$ of $T$ in the  moduli space $\cA_2(m)$ of principally polarized   abelian surfaces with full level-$m$ structures.\ 
Since $m\ge 3$, this moduli space is fine hence the inclusion $T'\subset \cA_2(m)$ defines a family $\cA\to T'$ of abelian surfaces with a relative principal polarization $\cH_\cA$ on $\cA$; since $m$ is even, there are sixteen sections corresponding to the $2$-torsion points in the fibers.

Let $\eps\colon\widehat\cA\to \cA$ be the blow up of the images of these sixteen sections.\ 
Multiplication by~$-1$  on $\cA$  lifts to $\widehat\cA$.\ 
Let $\pi\colon \widehat \cA\to \cK\coloneqq \widehat \cA/\pm1$ be the quotient map and let $\cE\subset \widehat\cA$ be the image by $\pi$ of the exceptional divisor  $\eps^{-1}(\cA[2])$.\ 
The divisor $\cE$ is the branch locus of $\pi$ hence its class is divisible by~$2$ in $\NS(\cK)$; it splits as the sum of sixteen irreducible divisors $\cE_1,\dots,\cE_{16}$.\ 
The relative polarization $\cH_\cA $ lifts to a relative quasi-polarization on $ \widehat\cA \to T'$ which induces a relative quasi-polarization $\cH$ on $  \cK\to T'$.

By Proposition~\ref{prop20}, there are relative polarizations on~$\cK$ of  all even degrees $\ge 320$.\ 
Method~(M1) of Section~\ref{sect23} also applies: using the $T'$-isomorphism $\cA\isomto \Pic^0(\cA/T')$ given by the principal polarization $\cH_\cA $ and a section of $\cA\to T'$ of order $2$, one can construct a double \'etale covering $\cA'\to\cA$ and proceed as in  footnote~\ref{f2}.\ We obtain relative polarizations on~$\cK$ of  all even degrees~$\ge 190$.

Finally, to use method~(M2), one needs global halves of  classes of the type $\cH-\cE_1-\dots- \cE_{6}$.\ 
To achieve this, one needs to take a suitable ramified double cover of the base curve~$T'$, as explained in ~\cite[Section~3]{bkps}.
\end{proof}

\begin{rema}\upshape\label{rem}
One can be more precise concerning the nature of the complete curves that we constructed in $\cF^0_{2e}$.\ 
For example, in the proof above, we can take for $T$ a complete Shimura curve.\ 
More precisely, let  $Q$ be an indefinite quaternion algebra over $\Q$ of reduced discriminant~$D$ and consider principally polarized abelian surfaces whose endomorphism rings contain a maximal order of $Q$.\ 
Their Picard number is at least $3$ and their  locus in the moduli space $\cA_2$ of principally polarized abelian surfaces only depends on $D$ and is the finite union of images of projective Shimura curves (defined over $\Q$; see \cite[Proposition~4.3]{rot}).\ 
When $D\in\{6,10\}$, these loci are irreducible and the Shimura curves are isomorphic to $\P^1$ over~$\Q$ (\cite[Proposition~7.1]{rot}).\footnote{The abelian surfaces that they parametrize were   described in \cite[Theorem~1.3]{hamu} (see also \cite[Section~7]{rot})) as the  Jacobians of explicit   genus-2 curves.}

During the proof above, we needed to take a ramified cover $T'\to T$ in order to work with a family $\cA\to T'$ of abelian surfaces.\ 
When $t'\in T'$, the Kummer variety $\Kum(\cA_{t'})$ only depends on the  image $t$ of $t'$ in $T$, but the polarization that we construct depends in most cases on an ordering of the $(-2)$-curves   $E_1,\dots,E_{16}$.\ 
However, if we consider primitive ample classes of the type
$$L=aH- \frac{c}{2}\sum_{i=1}^{16}E_i,$$
where $a$ and $c$ are integers, the image of the pair $(\Kum(\cA_{t'}),L)$ will only depend on the point~$t$ and the modular map $T'\to \cF^0_{2e}$ will factor through $T$.

For $L$ to be primitive, we need $\gcd(a,c)=1$ and for ampleness, we need  $a>2c$ (Corollary~\ref{prop2}).\ 
It follows that in degrees $2e=4a^2-8c^2$, for all positive integers $a$ and $c$ subject to these two conditions, we obtain rational curves in $\cF^0_{2e}$. 
\end{rema}


\section{Ample classes on hyper-K\"ahler manifolds}\label{sec:HK}

In this section, we use moduli spaces of stable sheaves/complexes on twisted K3 or abelian surfaces, more generally on CY2 categories, to obtain complete families of polarized hyper-K\"ahler manifolds in any dimension, once we start from a complete family of polarized CY2 categories.\ Our tool is a minor generalization of~\cite{muk2}, together with techniques from~\cite{bama1,bama2}.\
We start by   giving in Section~\ref{subsec:twisted} a short review of twisted K3 and abelian surfaces, which provide our main sets of examples.\
Our main result, Theorem~\ref{thm:families}, is stated  in Section~\ref{subsec:Mukai} and is a direct rewriting of the results of~\cite{families}.\
Finally, we discuss various examples in Section~\ref{subsec:examples}.

\subsection{Twisted K3 and abelian surfaces}\label{subsec:twisted}

We give  a short review of twisted K3 and abelian surfaces, following \cite{HuSt}.\
Let $S$ be a complex K3 surface.\
We define its \emph{Brauer group} as
\[
\Br(S) \coloneqq  H^2(S,\cO_S^*)_{\mathrm{tors}}.
\]
It parameterizes equivalence classes of Azumaya algebras over $S$.

\begin{defi}\label{def:TwistedK3}
A \emph{twisted K3 surface} is a pair $(S,\alpha)$, where $S$ is a K3 surface and $\alpha\in\Br(S)$.\
We denote by $\coh(S,\alpha)$ the abelian category of $\alpha$-twisted coherent sheaves on $S$ and by $\mathrm{D}^\mathrm{b}(S,\alpha)$ its bounded derived category.
\end{defi}

Using the exponential sequence on $S$, we see that for any $\alpha\in\Br(S)$, there exists a \emph{B-field} $B\in H^2(S,\Q)$ such that
\[
\exp(B^{0,2}) = \alpha.
\]
Such a B-field is unique modulo $H^2(S,\Z)$ and $\NS(S)\otimes\Q$.

\begin{defi}\label{def:HodgeStructure}
Let $(S,\alpha)$ be a twisted K3 surface and let $B\in H^2(S,\Q)$ be a B-field.\
We define a weight-2 polarized Hodge structure $\widetilde{H}(S,B,\Z)$ on the cohomology $H^*(S,\Z)$, endowed with the Mukai pairing, by setting
\[
\widetilde{H}^{2,0}(S,B) \coloneqq  \C  \exp(B)[\eta] = \C  \left( [\eta] + B\wedge [\eta] \right),
\]
where $[\eta]\in H^2(S,\C)$ is the class of a nondegenerate holomorphic 2-form $\eta$ on $S$.\
We denote by $H_{\mathrm{alg}}^*(S,B,\Z)$ its $(1,1)$-part, given by
\[
H_{\mathrm{alg}}^*(S,B,\Z) \coloneqq  \widetilde{H}^{1,1}(S,B) \cap H^*(S,\Z).
\]
\end{defi}

This Hodge structure does not depend  on the choice of $B$, up to noncanonical isomorphism.\
Fixing $B$,  we also have, by~\cite[Proposition 1.2]{HuSt}, a well-defined notion of Mukai vector
\[
v^B=\sqrt{\mathrm{td}_S}\cdot\ch^B\colon K(S,\alpha)\lra H_{\mathrm{alg}}^*(S,B,\Z),
\]
where $K(S,\alpha)$ denotes the Grothendieck group of the abelian category $\coh(S,\alpha)$.

\begin{exam}\label{ex:untwisted}
If $\alpha$ is trivial,   we can choose $B=0$  and   the Hodge structure $\widetilde{H}(S,B,\Z)$ coincides with the usual Mukai Hodge structure on $H^*(S,\Z)$ and the Mukai vector with the usual Mukai vector.
\end{exam}

Mukai's moduli theory works on twisted K3 surfaces  as in the untwisted case.\
We summarize the main results as follows.

Let $(S,\alpha)$ be a twisted K3 surface.\
We denote by $\Stab^\dagger(\Db(S,\alpha))$ the distinguished connected component of the space of Bridgeland stability conditions on $\Db(S,\alpha)$, with respect to the algebraic Mukai lattice $H_{\mathrm{alg}}^*(S,B,\Z)$, described in~\cite{bri2} (see also~\cite{hms1} for the twisted case).\ 
The original definition is in~\cite{bri1}; here we use~\cite[Definition~21.15]{families}, with the addition that moduli spaces exist.

Given $\sigma\in\Stab^\dagger(\Db(S,\alpha))$ and  a Mukai vector $\bv \in H_{\mathrm{alg}}^*(S,B,\Z)$, we denote by $M_{(S,\alpha),\sigma}(\bv )$   the moduli space of $\sigma$-semistable objects on $\Db(S,\alpha)$ with $B$-twisted Mukai vector $\bv $ (see~\cite{HuSt,yos2,tod,bama1,families}).

\begin{theo}[Yoshioka]\label{thm:Yoshioka}
Let $(S,\alpha)$ be a twisted K3 surface and let $B\in H^2(S,\Q)$ be a B-field.\ 
Let $\bv \in H_{\mathrm{alg}}^*(S,B,\Z)$ be a primitive Mukai vector and let $\sigma\in\Stab^\dagger(\Db(S,\alpha))$ be $\bv$-generic.
\begin{itemize}
\item[\textnormal{(a)}] The moduli space $M_{(S,\alpha),\sigma}(\bv )$ is nonempty if and only if $\bv^2+2\geq0$.\
In that case, $M_{(S,\alpha),\sigma}(\bv )$ is a smooth projective hyper-K\"ahler manifold of $\mathrm{K3}^{[n]}$-type, where $n=\frac{\bv^2+2}{2}$.
\item[\textnormal{(b)}] The Mukai isomorphism gives   isometries of Hodge structures
\begin{align*}
&\vartheta\colon \bv^\perp/\bv \isomlra H^2(M_{(S,\alpha),\sigma}(\bv ),\Z)  &\text{if } \bv^2=0,\\
&\vartheta\colon \bv^\perp \isomlra H^2(M_{(S,\alpha),\sigma}(\bv ),\Z) &\text{if } \bv^2\geq 2,
\end{align*}
where $\bv^\perp\subset H^*(S,B,\Z)$ is endowed with the induced sub-Hodge structure, while\break $H^2(M_{(S,\alpha),H}(\bv ),\Z)$ is endowed with the standard Hodge structure together with the Beauville--Bo\-go\-mo\-lov--Fujiki form.
\end{itemize}
\end{theo}

The $\bv$-genericity of the  stability condition $\sigma$ means $M_{(S,\alpha),\sigma}(\bv)=M_{(S,\alpha),\sigma}^{\mathrm{st}}(\bv)$, namely all $\sigma$-semistable objects are $\sigma$-stable.\

Given an ample divisor $H$ on $S$, we denote the classical moduli spaces (particular cases of the above theorem) of $H$-Gieseker semistable (resp.~stable) $\alpha$-twisted sheaves by $M_{(S,\alpha),H}(\bv)$ (resp.~$M_{(S,\alpha),H}^{\mathrm{st}}(\bv )$) and of slope-semistable (resp.~slope-stable) torsion-free sheaves by $M_{(S,\alpha),H}^{\mu}(\bv )$ (resp.~$M_{(S,\alpha),H}^{\mu\mathrm{-st}}(\bv )$).

\begin{rema}\label{rmk:abelian}
Analogously, one can define twisted abelian surfaces and their moduli spaces.\
Let $(A,\alpha)$ be a twisted abelian surface.\
Given $\sigma\in\Stab^\dagger(\Db(A,\alpha))$,\footnote{For twisted abelian surfaces, the space of stability conditions is   connected.}  an analogue of Theorem~\ref{thm:Yoshioka} holds for the \emph{generalized Kummer moduli space} $K_{(A,\alpha),\sigma}(\bv)$, namely the fiber at~$0$ of the Albanese morphism $\alb\colon M_{(A,\alpha),\sigma}(\bv)\to A\times A^\vee$.\
More precisely, let $\bv \in H_{\mathrm{alg}}^*(A,B,\Z)$ be a primitive Mukai vector.\ When  $\bv^2\geq6$, the moduli space $K_{(A,\alpha),\sigma}(\bv)$ is nonempty of dimension $\bv^2-2$ and the Mukai isomorphism   gives an isomorphism $\vartheta\colon  \bv^\perp\subset H^*(A,B,\Z) \isomto H^2(K_{(A,\alpha),\sigma}(\bv),\Z)$.

The case $\bv^2=4$ is slightly degenerate:  the generalized Kummer moduli space is isomorphic to a K3 surface.\
We still have the Mukai morphism, but it is not an isomorphism: we can identify $\bv^\perp$ with the part of the cohomology of $K_{(A,\alpha),H}(\bv)$ coming from the abelian surface plus an extra class.\
This identification is not an isometry in this case: it satisfies
\[
q(\vartheta(\mathbf{a}),\vartheta(\mathbf{b})) =2 (\mathbf{a},\mathbf{b}),
\]
for all $\mathbf{a},\mathbf{b}\in \bv^\perp$.
When $\bv=(1,0,-2)$ (and so $B=0$) and $\sigma$ is $\bv$-generic, then $K_{(A,\alpha),H}(\bv )\cong\Kum(A)$, the class $-(0,D_A,0)$ gets identified with the cohomology class $D\coloneqq \alpha(D_A)$ induced by $D_A$, while the class $-(1,0,2)$ gets identified with $\frac 12 \sum_{i=1}^{16} E_i$.

As in the K3 case, given an ample divisor $H$ on $A$, we use the notation $K_{(A,\alpha),H}(\bv)$, and so on, for the particular case of $H$-Gieseker semistable sheaves.
\end{rema}

\subsection{Polarized families of moduli spaces}\label{subsec:Mukai}

We follow \cite[Sections 2--5]{per} for the basic notions used in this section (see also~\cite[Section 2]{MaSt}).\
The goal is to give a short review of~\cite[Part IV]{families}: a polarized family of CY2 categories gives rise to a polarized family of moduli spaces of stable objects.\
We will apply our results only in geometric situations arising from CY2 categories of twisted K3 of abelian surfaces, as in Section~\ref{subsec:twisted}.

Let $\cD$ be a smooth proper CY2 category over the complex numbers, as in~\cite[Definition~6.1]{per}.\
We denote by $\widetilde{H}(\cD,\Z)$ its topological K-theory, together with the Mukai Hodge structure.\
We also have a Mukai vector $v\colon K(\cD)\to H_{\textnormal{alg}}(\cD,\Z)$ with the same formal properties as in the twisted K3 or abelian surface case.

\begin{exam}\label{ex:TiwstedK3vsCY2categories}
The main example we will consider  is when $(S,\alpha)$ is a twisted K3 or abelian surface and $\cD=\Db(S,\alpha)$.\
In these cases, $\widetilde{H}(\cD,\Z)$ is isometric to $\widetilde{H}(S,B,\Z)$ as Hodge structures, once a B-field lift is fixed.
\end{exam}

As in the twisted setting, for $\sigma=(Z,\cP)\in\Stab(\cD)$, the space of stability conditions with respect to the algebraic Mukai lattice $H_{\textnormal{alg}}(\cD,\Z)$, and $\bv \in H_{\mathrm{alg}}^*(\cD)$ a Mukai vector, we denote by $M_{\cD,\sigma}(\bv )$   the moduli space of $\sigma$-semistable objects on $\cD$ with Mukai vector $\bv$.\ It is a proper algebraic space over $\C$ (\cite{ap,ahlh}).\
The stable locus $M_{\cD,\sigma}^{st}(\bv)$ is open, smooth, and symplectic.\
The main result of \cite{bama1} shows that there exists a real numerical Cartier divisor class $\ell_{\sigma}(\bv)$ on $M_{\cD,\sigma}(\bv)$ which is strictly nef (it is ample in geometric situations, as we will see below).

We now consider the relative situation and introduce the notion of polarized family of~CY2 categories.\
Let $T$ be a  quasi-projective scheme over $\C$.\
Let $\cD/T$ be a CY2 category over~$T$ (still in the sense of~\cite[Definition~6.1]{per}).\
We denote by $\widetilde{H}(\cD/T,\Z)$ the Mukai local system, as in~\cite[Definition~6.4]{per}.\
The {\em uniformly numerical relative Grothendieck group of $\cD$ over $T$} (see ~\cite[Proposition and Definition 21.5]{families}) is denoted by $\cN(\cD/T)$; it is a free abelian group of finite rank which can be thought of as the numerical Grothendieck group of $\cD_t$, for a very general closed point $t\in T$, or equivalently as the group of sections of $\widetilde{H}(\cD/T,\Z)$ which are algebraic on each fiber.

The notion of stability condition on $\cD$ over $T$ was introduced in~\cite{families}.\
We denote by $\Stab_{\cN}(\cD/T)$ the space of stability conditions $\underline{\sigma}=\{\sigma_t\}_{t \in T}$ on $\cD$ over $T$ whose central charge factors through $\cN(\cD/T)$ (see~\cite[Theorem 22.2]{families}\footnote{In the notation of \emph{loc.~cit.}, this means that the finite-rank free abelian group $\Lambda$ is $\cN(\cD/T)$ and the morphism $v\colon K_{\mathrm{num}}(\cD/T)\to \Lambda$ is given by the Mukai vector.}).\
It leads to the following definition.

\begin{defi}\label{def:PolarizedFamilyCY2}
We say that a CY2 category $\cD/T$ over $T$ is a \emph{polarized family} if there exists a stability condition $\underline{\sigma}=\{\sigma_t\}_{t \in T}\in\Stab_{\cN}(\cD/T)$ on $\cD$ over $T$ whose central charge factors through $\cN(\cD/T)$.
\end{defi}

One of the main results in~\cite{families}, together with the generalization of~\cite{muk1} given in~\cite{per},  implies that a family of polarized CY2 categories gives a family of quasi-polarized moduli spaces.

\begin{theo}\label{thm:families}
Let $(\cD/T,\underline{\sigma})$ be a polarized CY2 category over a complex quasi-projective scheme $T$.\
Let $\bv\in\cN(\cD/T)$ be a Mukai vector for which the relative moduli $M_{\underline{\sigma}}(\bv)$ consists only of stable objects.\
Then $M_{\underline{\sigma}}(\bv)\to T$ is a smooth and proper algebraic space over $T$, all the fibers are projective and symplectic, and there exists a relative real numerical Cartier divisor class $\ell_{\underline{\sigma}}\in N^1(M_{\underline{\sigma}}(\bv)/T)$ that is relatively strictly nef.\
If $T$ is smooth,   $M_{\underline{\sigma}}(\bv)\to T$ is also projective.
\end{theo}

\begin{proof}
The fact that $M_{\underline{\sigma}}(\bv)\to T$ is proper   and the existence and positivity property of the divisor class $\ell_{\underline{\sigma}}$ is exactly \cite[Theorem~21.24 and Theorem~21.25]{families}.\
Since the relative moduli space~$M_{\underline{\sigma}}(\bv)$ consists only of stable objects, the smoothness over $T$ follows from \cite[Theorem~1.4]{per}.\
The projectivity of the fibers, or the more general statement if $T$ is smooth, then follows  from \cite[Corollary 3.4]{v-p}.\
Finally, the symplectic form comes from relative Serre duality: it is nondegenerate by assumption and skew-symmetric by~\cite{vdb}.
\end{proof}

The relative divisor class is ample when one of the fibers $\cD_t$ over a closed point $t\in T$ is geometric.

\begin{prop}\label{prop:AmpleFamily}
In addition to the  assumptions made above, let us further assume that there exists a closed point $t_0 \in T$ such that $\cD_{t_0}\cong \Db(S,\alpha)$, for a twisted K3 or abelian surface $(S,\alpha)$, and that $\sigma_{t_0}\in \Stab^{\dagger}(\Db(S,\alpha))$.\
Then $\ell_{\underline{\sigma}}$ is relatively ample.
\end{prop}

\begin{proof}
The argument is based on~\cite[Section 33]{families} and was found by Giulia Sacc\`a.\
We repeat it here for completeness.

We can slightly deform $\underline{\sigma}$ and, upon taking multiples of the divisor class $\ell_{\underline{\sigma}}$, assume that this class is integral.\
By~\cite[Corollary 7.5]{bama1}, the class $\ell_{\sigma_{t_0}}=\ell_{\underline{\sigma}}|_{M_{\sigma_{t_0}}(\bv)}$ is ample on $M_{\sigma_{t_0}}(\bv)$.\
Since ampleness is an open property, $\ell_{\sigma_{t}}$ is ample for all $t$ in a Zariski open subset~$U\subset T$.\
Let us fix a line bundle $L_{\underline{\sigma}}$ whose numerical class is $\ell_{\underline{\sigma}}$.\
By relative Serre vanishing, we can assume that $L_{\sigma_t}$ has no higher cohomology  for all $t \in U$.\
Hence
\[
h^0(M_{\sigma_t}(\bv), L_{\sigma_t}^{\otimes m})=\chi(M_{\sigma_t}(\bv), L_{\sigma_t}^{\otimes m}) 
\]
for all $t\in U$ and $m>0$, and thus this number is independent of $t$.

By semicontinuity, this shows that $h^0(M_{\sigma_t}(\bv), L_{\sigma_t}^{\otimes m})$ has maximal growth for all $t \in T$.\ Therefore it is big on $M_{\sigma_t}(\bv)$ for all $t\in T$.\ Since $M_{\sigma_t}(\bv)$ is a smooth, projective, symplectic variety, it has trivial canonical bundle.\ Hence, by Kawamata's Base Point Free Theorem, since the line bundle $L_{\sigma_{t}}$ is also strictly nef, it must be ample for all $t\in T$, as we wanted.
\end{proof}

\subsection{Examples}\label{subsec:examples}

We use Theorem~\ref{thm:families} to construct families of polarized HK manifolds.\
We consider separately the cases of twisted abelian and K3 surfaces: the ideas are similar but slightly more involved in the K3 case.

\subsubsection*{Generalized Kummer varieties}
Let $r,d\geq1$.\
Let $(A,\alpha,H_A)$ be a polarized twisted abelian surface, with $H_A^2=2d$ and $\alpha^r=\id$.\
We let $B\in H^2(A,\Q)$ be a B-field associated with $\alpha$ and set $B_0\coloneqq r B\in H^2(A,\Z)$.\
We also set $a\coloneqq H_A\cdot B_0\in\Z$ and $b\coloneqq \frac{B_0^2}{2}\in\Z$.

We further fix integers $c$ and $s$ and consider the Mukai vector  
\[
\bv \coloneqq (r,c H_A + B_0, s)\in H^*_{\mathrm{alg}}(A,B,\Z).
\]
We assume  
\[
\bv^2= 2(dc^2 + ac+ b- rs) \geq 4\qquad\text{ and }\qquad \gcd(r,2cd+a)=1.
\]

We consider the Mukai vectors $\boldsymbol{\ell},\boldsymbol{\delta}\in\bv^\perp\subset H^*_{\mathrm{alg}}(A,B,\Z)$ given by
\[
\boldsymbol{\ell}\coloneqq -(0,r H_A,2cd+a) \qquad\text{ and }\qquad \boldsymbol{\delta}\coloneqq -r \bv -( \bv^2 ) \mathbf{w},
\]
where $\mathbf{w}\coloneqq (0,0,1)$.

Finally, we let $K\coloneqq K_{(A,\alpha),H_A}(\bv)$ be the generalized Kummer variety (of dimension $\bv^2-2$) arising from the moduli space of $H$-Gieseker stable sheaves (see Remark~\ref{rmk:abelian}).\
On $K$, the divisor class
\[
D_u \coloneqq  \vartheta\left( u \boldsymbol{\ell} - \boldsymbol{\delta} \right) \in \NS(K)_{\R}
\]
is ample for all $u$ sufficiently large (\cite[Corollary~9.14]{bama1}).\
The following result gives an explicit lower bound.

\begin{prop}\label{prop:AmpleKummer}
In the above notation, the class $D_u$ on $K$ is ample for all $u>r \frac{\bv^2}{2}$.\
Moreover, if $\bv^2\leq 2r-2$, the class $D_{\infty}\coloneqq \vartheta(\boldsymbol{\ell})$ is also ample.
\end{prop}

\begin{proof}
To prove that $D_u$ is ample for all $\infty>u>r \frac{\bv^2}{2}$, we have to show
\begin{equation}\label{eq:AmpleKummer1}
\frac{(\mathbf{a},\boldsymbol{\delta})}{(\mathbf{a},\boldsymbol{\ell})}\stackrel{\mathord{?}}{\leq} r \frac{\bv^2}{2} 
\end{equation}
for all Mukai vectors $\mathbf{a} = (\alpha, D, \beta)\in H^*_{\mathrm{alg}}(A,B,\Z)$ satisfying $(\mathbf{a},\boldsymbol{\ell})\ne 0$,
\begin{equation}\label{eq:AmpleKummer3}
   \mathbf{a}^2 \geq 0, \qquad \text{ and }\qquad 1\leq (\mathbf{a},\bv)\leq \frac{\bv^2}{2}   
\end{equation}
(we may of course assume further $ (\mathbf{a},\boldsymbol{\ell})  (\mathbf{a},\boldsymbol{\delta})\geq1$).\footnote{This is the version for abelian surfaces of~\cite[Theorem 12.1]{bama2}.\ 
Although this result is not explicitly stated in~\cite{myy,yy,yos3}, it   follows directly from~\cite[Proposition 3.15]{yos3} or by a similar argument as in~\cite{bama2}.}

Explicitly, setting $\widetilde{D}\coloneqq r D - \alpha (c H_A+r B) $, we have
\[
\begin{split}
&(\mathbf{a},\boldsymbol{\ell}) = H_A \cdot \widetilde{D} ,\\
&(\mathbf{a},\boldsymbol{\delta})^2 = \bv^2 \widetilde{D}^2 + r^2 \left( (\mathbf{a},\bv)^2 - \bv^2 \mathbf{a}^2\right).
\end{split}
\]
Hence, \eqref{eq:AmpleKummer1} can be written as
\begin{equation*}\label{eq:AmpleKummer2}
\bv^2 \widetilde{D}^2 + r^2 \left( (\mathbf{a},\bv)^2 - \bv^2 \mathbf{a}^2\right) \stackrel{\mathord{?}}{\leq} r^2 \frac{(\bv^2)^2}{4} (H_A\cdot \widetilde{D})^2.
\end{equation*}
If $\widetilde{D}^2\leq0$, this inequality   follows  from the inequalities~\eqref{eq:AmpleKummer3}.

Hence, we may assume $\widetilde{D}^2\geq2$.\
By the Hodge Index Theorem, we have $(H_A\cdot \widetilde{D})^2 \geq H_A^2 \widetilde{D}^2$ hence   it is enough to show the inequality
\begin{equation*}
\bv^2 \widetilde{D}^2 + r^2 \left( (\mathbf{a},\bv)^2 - \bv^2 \mathbf{a}^2\right) \stackrel{\mathord{?}}{\leq} r^2 \frac{(\bv^2)^2}{4} H_A^2 \widetilde{D}^2.
\end{equation*}
Upon dividing by $\bv^2\geq4$, this is equivalent to
\begin{equation}\label{eq:AmpleKummer4}
\widetilde{D}^2 + r^2 \left( \frac{(\mathbf{a},\bv)^2 }{\bv^2}- \mathbf{a}^2\right) \stackrel{\mathord{?}}{\leq} r^2 \frac{\bv^2}{4} H_A^2 \widetilde{D}^2.
\end{equation}
By~\eqref{eq:AmpleKummer3} again, we have
$$\frac{(\mathbf{a},\bv)^2 }{\bv^2}- \mathbf{a}^2\le \frac{(\mathbf{a},\bv)^2 }{\bv^2}\le \frac{\bv^2 }{4},$$
so that~\eqref{eq:AmpleKummer4} is implied by  the inequality
\begin{equation}\label{eq:AmpleKummer5} 
r^2 \frac{\bv^2}{4}\stackrel{\mathord{?}}{\leq} \widetilde{D}^2 \Bigl( r^2\frac{\bv^2}{4} H_A^2 - 1\Bigr) .
\end{equation}
Since $\widetilde{D}^2\geq2$, $r\geq1$, $\bv^2\geq4$, and $H_A^2\geq2$, the right side of~\eqref{eq:AmpleKummer5} is 
$$\ge 2 \Bigl( r^2\frac{\bv^2}{4} H_A^2 - 1\Bigr)\ge r^2\bv^2 - 2= r^2 \frac{\bv^2}{4}+3 r^2 \frac{\bv^2}{4} - 2\ge  r^2 \frac{\bv^2}{4}+3  - 2>r^2\frac{\bv^2}{4},$$
which implies~\eqref{eq:AmpleKummer5}.

To show that $D_{\infty}$ is ample, we only have to show that there are no Mukai vectors $\mathbf{a}$ as above such that $(\mathbf{a},\boldsymbol{\ell})=0$.\
This can be done by a direct computation.\
We use instead a quicker, more geometric, argument.\
The divisor class $D_{\infty}$ is associated with the Donaldson--Ulhenbeck--Yau compactification: it induces the birational morphism
\[
K_{(A,\alpha),H_A}(\bv)\lra K_{(A,\alpha),H_A}^{\mu}(\bv).
\]
To show that $D_{\infty}$ is ample, it is enough to show that all slope-semistable torsion-free sheaves are actually slope-stable vector bundles.

To this end, we first note the equality
\[
H_A\cdot (c H_A + r  B) = 2cd + a.
\]
Hence, since $\gcd(r,2cd+a)=1$, there are no properly slope-semistable torsion-free sheaves.\
The assumption $\bv^2\leq 2r-2$ then guarantees  that all torsion-free sheaves are actually vector bundles: indeed, the Mukai vector $\bv'$ of the double dual would satisfy $(\bv')^2<0$, which is impossible.
\end{proof}

In the special case $r=1$, we obtain the following generalization of Corollary~\ref{prop2}.\
Let $n\geq1$ and set $\bv\coloneqq (1,0,-n-1)$, so that $K$  is isomorphic to the generalized Kummer variety $\Kum_n(A)$.\
Explicitly, we have
\[
\boldsymbol{\ell}=-(0,H_A,0)\qquad\textnormal{and}\qquad \boldsymbol{\delta}\coloneqq -(1,0,n+1).
\]
Then the class $H_n\coloneqq D_{\infty}=\vartheta(\boldsymbol{\ell})$ is big and nef, but not ample.\
We let $\delta\coloneqq \vartheta(\boldsymbol{\delta})$; it is half the class of the  restriction of the divisor on the Hilbert scheme parameterizing nonreduced subschemes. 

\begin{coro}\label{cor:AmpleKummer}
On any generalized Kummer variety $\Kum_n(A)$, the class $a H_n - \delta$ is ample for all real numbers $a>n+1$.
\end{coro}

Theorem~\ref{thm:Kum} immediately follows  from Corollary~\ref{cor:AmpleKummer}.

Recall from Remark~\ref{rmk:abelian} that if $n=1$, then $\Kum_1(A)=\Kum(A)$, the notation for $H=H_1$ is coherent with above, and $\delta = \frac 12 \sum_{i=1}^{16} E_i$.\
Hence, the statement is exactly Corollary~\ref{prop2}.\
For $n\geq2$, it is optimal since $\NS(\Kum_n(A))\cong \NS(A)\oplus \Z\cdot \delta$.\
Moreover, if there exists a divisor~$D$ with $D^2=0$ and $D\cdot H_A=1$, we get a Mukai vector $\mathbf{a}=(1,D,0)$ with $\mathbf{a}^2=0$ and $(\mathbf{a},\bv)=n+1$.\
Thus the class $(n+1) H_n - \delta$ is nef but not ample on $\Kum_n(A)$ in that case.

\subsubsection*{K3-type}
Let $r,e\geq1$.\
Let $(S,\alpha,H)$ be a polarized twisted K3 surface, with $H^2=2e$ and $\alpha^r=\id$.\
We let $B\in H^2(S,\Q)$ be a B-field associated with $\alpha$ and set $B_0\coloneqq r B\in H^2(S,\Z)$.\
We also set $a\coloneqq H\cdot B_0\in\Z$ and $b\coloneqq \frac{B_0^2}{2}\in\Z$. 

We further fix integers $c$ and $s$ and consider the Mukai vector  
\[
\bv \coloneqq (r,c H +    B_0, s)\in H^*_{\mathrm{alg}}(S,B,\Z).
\]
We assume  
\[
\bv^2 =2(er^2+ ac+ b- rs)\geq 0\qquad\text{ and }\qquad \gcd(r,2ce+a)=1.
\]

We set $\mathbf{w}\coloneqq (0,0,1)$ and consider the Mukai vectors $\boldsymbol{\ell},\boldsymbol{\delta}\in\bv^\perp\subset H^*_{\mathrm{alg}}(S,B,\Z)$ given by
\[
\boldsymbol{\ell}\coloneqq -(0,r H,2ce+a) \qquad\text{ and }\qquad \boldsymbol{\delta}\coloneqq - r \bv -( \bv^2 ) \mathbf{w} .
\]

Finally, we let $M\coloneqq M_{(S,\alpha),H}(\bv)$ be the moduli space (of dimension $\bv^2+2$) of $H$-Gieseker stable sheaves on $S$ with Mukai vector $\bv$.\
On $M$, the divisor class
\[
D_u \coloneqq  \vartheta\left( u \boldsymbol{\ell} - \boldsymbol{\delta} \right) \in \NS(M)_{\R}
\]
is ample for all $u$ sufficiently large.\
The following result is the analog of Proposition~\ref{prop:AmpleKummer} for K3 surfaces.

\begin{prop}\label{prop:AmpleHilbert}
With the  notation above, the divisor class  $D_u$ on $M$ is ample for all
\[
u> r \sqrt{\frac{(\bv^2)^2}{4}+2\bv^2}.
\]
Moreover, if $\bv^2\leq 2r-4$, the divisor class $D_{\infty}\coloneqq \vartheta(\boldsymbol{\ell})$ is also ample.
\end{prop}

\begin{proof}
The proof goes along the same lines as the proof of Proposition~\ref{prop:AmpleKummer}.\
The only difference is that, to show that $D_u$ is ample, we use~\cite[Theorem 12.1]{bama2}.\
Thus, we are looking at Mukai vectors $\mathbf{a} = (\alpha, D, \beta)\in H^*_{\mathrm{alg}}(S,B,\Z)$ satisfying
\[
\mathbf{a}^2 \geq -2 \qquad \text{ and }\qquad 1\le (\mathbf{a},\mathbf{v})\le \frac{\bv^2}{2} .
\]
This is the reason why the estimate on $u$ is more complicated, but the rest of the proof is analogous.
\end{proof}

As in the generalized Kummer case, by considering the case  $r=1$ (and so $B=0$) and~$c=0$ in Proposition~\ref{prop:AmpleHilbert}, we immediately obtain   Theorem~\ref{thm:Hilb}.\ 
More precisely, by taking $\bv=(1,0,1-n)$, we have $\boldsymbol{\ell}=-(0,H,0)$, $\boldsymbol{\delta}=-(1,0,n-1)$, $H_n=\vartheta(\boldsymbol{\ell})$, $\delta=\vartheta(\boldsymbol{\delta})$, and thus we obtain that the class $u  H_n-\delta$ is ample for $u>\sqrt{(n-1)^2+4(n-1)}$, as we wanted.

\subsubsection*{Explicit polarized families}

To apply Theorem~\ref{thm:families}, we need to make sure that starting from a polarized family of K3 or abelian surfaces, we can take a twist along the whole family, after taking a finite cover.\
This is the content of the following lemma.

\begin{lemm}\label{lem:TwistedFamily}
Let $T$ be an integral quasi-projective scheme and let $\cS\to T$ be a smooth family of K3 or abelian surfaces.\
Let $t_0\in T$ be a closed point and let $\alpha_{t_0}\in\Br(\cS_{t_0})$ be a Brauer class.\
Then there exist  an integral quasi-projective scheme $T'$, a generically finite projective surjective morphism $\phi\colon T'\twoheadrightarrow T$, a flat family of Azumaya algebras $\cA_{T'}$ on the base change $\cS_{T'}\to T'$, and a closed point $t_0'\in T'$ with $\phi(t_0')=t_0$, such that the class of the Azumaya algebra $\cA_{T'}|_{t_0'}$ is exactly $\alpha_{t_0}$.
\end{lemm}

\begin{proof}
Let us briefly sketch a proof of this well-known fact.\footnote{The general theory of moduli spaces of polarized twisted K3 surfaces has recently been studied in~\cite{bra}.}\
We consider the relative moduli space $f\colon \cT \to T$ of Azumaya algebras over $T$.\
On the special fiber at $t_0\in T$, we can choose an Azumaya algebra $\cA_{t_0}$ of minimal rank with class $\alpha_{t_0}$.\
Hence, $\cA_{t_0}$ is a slope-polystable vector bundle with trivial first Chern class, so it deforms along over any complex deformation of the K3 or abelian surface.\
This shows that $f$ is dominant in a neighborhood of~$\cA_{t_0}$.\
The scheme~$T'$ is then given by a general multisection of $f$ through $\cA_{t_0}$.
\end{proof}

\begin{exam}\label{ex:PolarizedK3}
To show that $\cF^0_{2e}$ contains a complete curve for $e\in\{18, 32, 36, 50, 54\}$, we can apply Proposition~\ref{prop:AmpleKummer} to the divisor $D_{\infty}$, together with Lemma~\ref{lem:TwistedFamily}.\
Explicitly, we consider a polarized twisted abelian surface $(A,\alpha, H_A)$ with $H_A^2=2d$, and we take $\bv=(r, B_0,0)$, where $B_0\coloneqq rB$ satisfies $a\coloneqq H\cdot B_0=1$ and $b\coloneqq \frac{B_0^2}{2}=2$.\ Then  $\bv^2=4$, the moduli space $K$  is a K3 surface, and  the divisor class $D_{\infty}$ is ample when $r\ge 3$, with self-intersection $2e\coloneqq 4dr^2$ (Proposition~\ref{prop:AmpleKummer}).\ 
Then,
\begin{itemize}
    \item taking $d=1$ and  $r=3$, we obtain $e=18$;
    \item taking $d=1$ and  $r=4$, we obtain $e=32$;
    \item taking $d=2$ and  $r=3$, we obtain $e=36$;
    \item taking $d=1$ and  $r=5$, we obtain $e=50$;
    \item taking $d=3$ and  $r=3$, we obtain $e=54$.
\end{itemize}
\end{exam}

\begin{exam}[Mukai, O'Grady]\label{ex:MukaiGeneralized}
Consider a polarized K3 surface $(S,H)$ of degree $2e$ and a Brauer class $\alpha\in\Br(S)$ of order $r\geq2$.\
We further assume that there is a B-field $B\in H^2(S,\Q)$ with $B_0\coloneqq r B\in H^2(S,\Z)$  associated with $\alpha$ such that $a\coloneqq H\cdot B_0=1$ and $2b\coloneqq B_0^2=2r$.\
If we consider the Mukai vector $\bv\coloneqq (r,B_0,1)$, the moduli space $M\coloneqq M_{(S,\alpha),H}(\bv)$ is a smooth projective K3 surface and the divisor $D_{\infty}$ is ample on $M$ of degree $2r^2e$ (Proposition~\ref{prop:AmpleHilbert}).

Given a polarized K3 surface $(S,H)$ of degree $2e$, we can always find a Brauer class $\alpha\in\Br(S)$ satisfying the above conditions.\
By applying Lemma~\ref{lem:TwistedFamily}, we obtain a diagram of finite morphisms
\[
\xymatrix{
& \cF_{e,r}'\ar[dr]^{g}\ar[dl]_{f} & \\
\cF_{2e}^0 && \cF_{2r^2e}^0,
}
\]
where $f$ is surjective and $g$ is dominant.\
In particular, given a complete subvariety in $\cF_{2e}^0$, this gives us a complete subvariety in $\cF_{2r^2e}^0$, for all $r\geq2$.
The morphism $g$ is not surjective: it misses some irreducible components of the Heegner divisor associated with a class of square~$-2r^2$.

These morphisms and their degrees can be better studied by using lattice theory and period domains.
Indeed, as proved in~\cite[Appendix]{og1}, the morphism $g$ is an open embedding and the morphism $f$ can be extended to quasi-polarized K3 surfaces to give a finite surjective morphism
\[
f\colon \cF_{2r^2e} \longrightarrow \cF_{2e}
\]
still denoted by $f$.\ At the level of moduli spaces, this map can be interpreted as follows.\ For a polarized K3 surface $(S,H)$ in $\cF_{2r^2e}^0$, we can consider the Mukai vector $\bv'=(r,H,re)$, the moduli space $N\coloneqq N_{(S,H),H}(\bv')$ (which is   not fine in general), and the line bundle on~$N$ with class $-\vartheta(1,0,-e)$.\
The degree of the morphism $f$ can also be studied in some cases (see~\cite{og2,kon}).
\end{exam}


\appendix\section{Numerical computations}\label{app}

We list here the various numerical statements that we used in the course of this article.\ The corresponding codes are available at

\url{https://www.imo.universite-paris-saclay.fr/~macri/Kummer_sage.pdf}

\subsection{}\label{st1} {\em All integers $36\le n\le 101$ can be written as the sum of the squares of fifteen positive integers that are all $\le \sqrt{\frac{n}3-3}$.}

\subsection{}\label{st2} {\em All integers $163\le e\le 1215$ or in the set $\{34$, $53$, $79$, $97$, $101$, $103$, $107$, $109$, $113,119$, $125$, $131$, $135$, $137$, $139$, $143$, $145$, $149$, $151$, $155$, $157$, $161\}$
can be written as $$ e=2 a^2 -\sum_{i=1}^{15}a_i^2-1,$$ with $a,a_1,\dots,a_{15}\in\Z$, $a_1\ge \dots\ge a_{15}\ge1$, and $a> a_{1}+a_{2}+a_{3}+a_{4}$.}

\subsection{}\label{st3} {\em All integers $32\le n\le 186$ or in the set $\{
9$, $15$, $16$, $22$, $23$, $24$, $25$, $30 \}$ can be written as $$ n=a^2 -\sum_{i=1}^{15}\binom{a_i}{2},$$ with $a,a_1,\dots,a_{15}\in\Z$, $a_1\ge \dots\ge a_{15}\ge1$, and $a> a_{1}+a_{2}+a_{3}+a_{4}-2$.}

\subsection{}\label{st4}  
{\em All integers  $95\le e\le 159$ or in the set
$\{38$, $57$, $59$, $71$, $73$, $75$, $77$, $79$, $81$, $83$, $85 \}$
can be written as 
$$e= 2a^2 - \sum_{i=1}^{15}a_i^2-\frac14, $$
where $a\in\Z$, $a_1,\dots,a_{15}\in\frac12\Z_{>0} $ and exactly seven of them are not integers, $a_1\ge \dots\ge a_{15}$, and $a>a_{1}+a_{2}+a_{3}+a_{4}$.}

\subsection{}\label{st5}  
{\em All integers $e\in\{
29$, $ 45$,   $ 47$, $ 49$, $ 63$, $ 65$, $ 67$, $ 69$, $87$, $89$, $91$, $93
\}$ can be written as 
$$e= 2  \Bigl(a+\frac12\Bigr)^2- \sum_{i=1}^{15}a_i^2-\frac14,$$
where $a\in\Z$,  $a_1,\dots,a_{15} \in\frac12\Z_{>0} $  and exactly five of them are not  integers, $a_1\ge \dots\ge a_{15}$, and $a\ge a_{1}+a_{2}+a_{3}+a_{4}$.}



\begin{thebibliography}{BHPV}

\bibitem[AP]{ap} Abramovich, D., Polishchuk, A., Sheaves of {$t$}-structures and valuative criteria for stable complexes, {\it J.~reine angew.~Math.} {\bf 590} (2006), 89--130.

\bibitem[AHLH]{ahlh} Alper, J., Halpern-Leistner, D., Heinloth, J., Existence of moduli spaces for algebraic stacks, eprint {\tt arXiv:1812.01128}.

\bibitem[BHPV]{bpv} Barth, W., Hulek, K., Peters, Ch., Van de Ven, A., {\it Compact complex surfaces,} Second edition. Ergebnisse der Mathematik und ihrer Grenzgebiete {\bf4}, Springer-Verlag, Berlin, 2004.

\bibitem[BL+]{families} Bayer, A., Lahoz, M., Macr\`i, E., Nuer, H., Perry, A., Stellari, P., Stability conditions in family, {\it Publ. Math. IH\'ES} {\bf 133} (2021), 157--325.

\bibitem[BM1]{bama1} Bayer, A., Macr\`i, E., Projectivity and birational geometry of Bridgeland moduli spaces, {\it J. Amer. Math. Soc.}  {\bf 27} (2014), 707--752.

\bibitem[BM2]{bama2} \bysame, MMP for moduli of sheaves on K3s via wall-crossing: nef and movable cones, Lagrangian fibrations, {\it Invent.~Math.}  {\bf 198} (2014),  505--590.

\bibitem[B]{ben} Benoist, O., Espaces de modules d'intersections compl\`etes lisses, Th\`ese, 2012, available at {\tt https://www.math.ens.fr/$\sim$benoist/articles/These.pdf}

\bibitem[BKPS]{bkps}  Borcherds, R., Katzarkov, L., Pantev, T., Shepherd-Barron, N.I., Families of K3 surfaces, {\it J. Algebraic Geom.}  {\bf 7} (1998), 183--193. 

\bibitem[Br]{bra} Brakkee, E., Moduli spaces of twisted K3 surfaces and cubic fourfolds, {\it Math. Ann.} {\bf 377} (2020), 1453--1479.

\bibitem[Bri1]{bri1} Bridgeland, T., Stability conditions on triangulated categories, {\it Ann. of Math.} {\bf 166} (2007), 317--345.

\bibitem[Bri2]{bri2} \bysame, Stability conditions on K3 surfaces, {\it Duke Math.\ J.} {\bf 141} (2008), 241--291.

\bibitem[DL]{dole} Dolgachev, I., Lehavi, D., On isogenous principally polarized abelian surfaces, in {\em  Curves and abelian varieties,} 51-–69, Contemp. Math. {\bf 465}, Amer. Math. Soc., Providence, RI, 2008.

\bibitem[GS1]{gs}  Garbagnati, A., Sarti, A., Projective models of K3 surfaces with an even set, {\it Adv. Geom.}  {\bf8} (2008), 413--440.

\bibitem[GS2]{gs2}  \bysame, Kummer surfaces and K3 surfaces with $(\Z/2\Z)^4$ symplectic action, {\it Rocky Mountain J. Math.}  {\bf46} (2016),  1141--1205.

\bibitem[vdGK]{gk} van der Geer, G., Katsura, T., Note on tautological classes of moduli of K3 surfaces, {\it Mosc. Math. J.} {\bf5} (2005), 775--779, 972.

\bibitem[GH]{gh} Griffiths, P., Harris, J., {\it Principles of algebraic geometry,} Pure and Applied Mathematics, Wiley-Interscience, New York, 1978.

\bibitem[HM]{hamu} Hashimoto, K., Murabayashi, N., Shimura curves as intersections of Humbert surfaces and defining equations of QM-curves of genus two, {\it T\^ohoku Math. J.} {\bf47} (1995), 271--296.

\bibitem[H]{huy} Huybrechts, D., {\it Lectures on K3 surfaces,} Cambridge Studies in Advanced Mathematics {\bf158},  Cambridge University Press, 2016.

\bibitem[HMS]{hms1} Huybrechts, D.. Macr\`i, E., Stellari, P., Stability conditions for generic K3 categories, {\it Compos. Math.} {\bf 144} (2008), 134--162.

\bibitem[HS]{HuSt} Huybrechts, D., Stellari, P., Equivalences of twisted K3 surfaces, {\it Math. Ann.} {\bf 332} (2005), 901--936.

\bibitem[K]{kon} Kond\={o}, S., On the Kodaira dimension of the moduli space of K3 surfaces, {\it Compos. Math.} {\bf 89} (1993), 251--299.

\bibitem[MS]{MaSt} Macr\`i, E., Stellari, P., Lectures on non-commutative K3 surfaces, Bridgeland stability, and moduli spaces, in {\it Birational geometry of hypersurfaces}, 199--265, Lect. Notes Unione Mat. Ital. {\bf 26}, Springer, 2019.

\bibitem[MYY]{myy} Minamide, H., Yanagida, S., Yoshioka, K., Some moduli spaces of Bridgeland's stability conditions, {\it Int. Math. Res. Not. IMRN} {\bf 19} (2014), 5264--5327.

\bibitem[M1]{muk1} Mukai, S., Symplectic structure of the moduli space of sheaves on an abelian or K3 surface, {\it Invent. Math.} {\bf 77} (1984), 101--116.

\bibitem[M2]{muk2} \bysame, Duality of polarized K3 surfaces, in {\it New trends in algebraic geometry (Warwick, 1996),} 311--326, London Math. Soc. Lecture Note Ser. {\bf 264}, Cambridge Univ. Press, 1999.

\bibitem[NZ]{nz}  Niven, I., Zuckerman, H., {\it An introduction to the theory of numbers,} 3rd edition, John Wiley \& Sons, Inc., New York-London-Sydney, 1972.

\bibitem[O1]{og2} O'Grady, K., Moduli of abelian and K3 surfaces, Ph.D.~Thesis, Brown University, 1986.

\bibitem[O2]{og1} \bysame, On the Kodaira dimension of moduli spaces of abelian surfaces, {\it Compos. Math.} {\bf 72} (1989), 121--163.

\bibitem[P]{per} Perry, A., The integral Hodge conjecture for two-dimensional Calabi-Yau categories, eprint {\tt arXiv:2004.03163}.

\bibitem[R]{rot} Rotger, V., Shimura curves embedded in Igusa's threefold, in {\it Modular curves and Abelian varieties,} 263--273, Progress in Mathematics {\bf 224}, Birkh\"auser, 2003. 

\bibitem[OEIS]{oeis}  Sloane, N. and The OEIS Foundation Inc.,
{\it The on-line encyclopedia of integer sequences},
2020, {\tt http://oeis.org}

\bibitem[T]{tod} Toda, Y., Moduli stacks and invariants of semistable objects on K3 surfaces, {\it Adv. Math.} {\bf 217} (2008), 2736--2781.

\bibitem[vdB]{vdb} Van den Bergh, M., The signs of Serre duality, appendix to: Bocklandt, R., Graded Calabi Yau algebras of dimension 3, {\it J.~Pure Appl.~Algebra} {\bf 212} (2008), 14--32.

\bibitem[V-P]{v-p} Villalobos-Paz, D., Moishezon Spaces and Projectivity Criteria, eprint {\tt arXiv:2105.14630}.

\bibitem[YY]{yy} Yanagida, S., Yoshioka, K., Bridgeland's stabilities on abelian surfaces, {\it Math. Z.} {\bf 276} (2014), 571--610.

\bibitem[Y1]{yos2} Yoshioka, K., Moduli spaces of twisted sheaves on a projective variety, in {\it Moduli spaces and arithmetic geometry}, 1--30, Adv. Stud. Pure Math. {\bf 45}, Math. Soc. Japan, 2006.

\bibitem[Y2]{yos3} \bysame, Bridgeland's stability and the positive cone of the moduli spaces of stable objects on an abelian surface, in {\it Development of moduli theory--Kyoto 2013}, 473--537, Adv. Stud. Pure Math. {\bf 69}, Math. Soc. Japan, 2016.
 
\end{thebibliography}
\end{document}